\documentclass[12pt, reqno]{amsart}
\usepackage{amsmath, amsthm, amscd, amsfonts, amssymb, graphicx, color,tikz}
\usepackage[bookmarksnumbered, colorlinks, plainpages]{hyperref}
\usepackage{mathrsfs}
\usepackage[utf8]{inputenc}
\usepackage{enumerate}

\newtheorem{theorem}{Theorem}[section]
\newtheorem{lemma}[theorem]{Lemma}
\newtheorem{proposition}[theorem]{Proposition}
\newtheorem{corollary}[theorem]{Corollary}
\theoremstyle{definition}

\newtheorem{remark}[theorem]{Remark}

\numberwithin{equation}{section}

\newcommand{\veps}{\varepsilon}
\newcommand{\CC}{\mathbb C}
\newcommand{\RR}{\mathbb R}
\newcommand{\DD}{\mathbb D}

\newcommand{\TT}{\mathbb T}

\newcommand{\Bd}{\mathbb B_d}
\newcommand{\Sd}{\mathbb S_d}
\newcommand{\mub}{\mu_{\mathbb B}}
\newcommand{\mus}{\mu_{\mathbb S}}

\allowdisplaybreaks[4]

\begin{document}
\setcounter{page}{1}

\title[Essential norms]
{Essential norms of composition operators and multipliers acting between different Hardy spaces}
\date{\today}

\author[F. Bayart]{Frédéric Bayart}

\address{Laboratoire de Math\'ematiques Blaise Pascal UMR 6620 CNRS, Universit\'e Clermont Auvergne, Campus universitaire des C\'ezeaux, 3 place Vasarely, 63178 Aubi\`ere Cedex, France.}
\email{frederic.bayart@uca.fr}


\subjclass[2010]{}

\keywords{}

\begin{abstract}
We compute the essential norm of inclusion operators, composition operators 
and multipliers acting from a closed subspace of some $L^p$-space into a subspace of some $L^q$-space, with $p>q.$
\end{abstract}
\maketitle

\section{Introduction}

\subsection{General context} Let $(\Omega_1,\mathcal E,\mu)$ and $(\Omega_2,\mathcal F,\nu)$
be two measure spaces, let $p,q\in[1,+\infty]$, let $X_p,Y_q$ be two closed subspaces of $L^p(\Omega_1)$ (resp. $L^q(\Omega_2)$)
and let $T_\varphi:X_p \to Y_q$ be a linear map depending on some ``symbol'' $\varphi$. Our aim in this paper is to obtain
estimates of the essential norm of $T_\varphi$ by quantities depending only on the symbol $\varphi.$ To emphasize that we work
with different values of $p$ and $q$, we will denote $\|T\|_{p\to q}$ (resp. $\|T\|_{e,p\to q}$) the norm (resp. the essential norm)
of any operator $T:X_p\to Y_q$. In particular, we will be concerned with composition operators and multiplication operators.

\subsection{Composition operators}
Let $\varphi$ be a holomorphic selfmap of the unit disc $\mathbb D$ and let $C_\varphi(f)=f\circ\varphi$ be the associated composition operator.
 Let also $p,q\in[1,+\infty]$. The characterization of compact composition operators $C_\varphi:H^p(\mathbb D)\to H^q(\mathbb D)$
and the computation of the essential norm $\|C_\varphi\|_{e,p\to q}$ have been investigated by many mathematicians
(see for instance \cite{Sha87}, \cite{Choe92}, \cite{GorMacCluer04} or \cite{CZ07} and the references therein).
In particular, the case $p\leq q$ is fairly well understood and $\|C_\varphi\|_{e,p\to q}$ is estimated
by quantities depending only on $\varphi$ and involving either Nevanlinna counting functions
or Carleson measures or integrals.

The case $p>q\geq 1$ remains more mysterious. H. Jarchow and T. Gobeler have shown (\cite{Jar98,Goe01})
that $C_\varphi$ is compact iff $E=E_\varphi=\{\xi\in\mathbb T:\ |\varphi^*(\xi)|=1\}$ has (Lebesgue) measure $0$,
where $\varphi^*$ denotes the radial limit function of $\varphi$. Upper and lower estimates for $\|C_\varphi\|_{e,p\to q}$
have been obtained in \cite{GorMacCluer04} when $q>1$ and generalized to $q=1$ in \cite{Dem11} but they do not coincide.

Our first main result in this paper is to get an estimation for $\|C_\varphi\|_{e,p\to q}$ in the spirit of what has been done in the
case $p\leq q$. Thus assume that $\sigma(E)>0$ where $\sigma$ is the normalized Lebesgue measure on the circle.
The map $\varphi^*_{|E}:E\to\varphi^*(E)$ is a non singular transformation from $(E,\sigma)$ into $(\varphi^*(E),\sigma)$
meaning that it does not collapse a set of positive measure into a set of measure $0$. We shall denote by $F_\varphi$ the 
Radon-Nikodym derivative of $\sigma_{|E}\circ(\varphi^*)^{-1}_{|\varphi^*(E)}$ with respect to $\sigma_{|\varphi^*(E)}$. It turns out that $\|C_\varphi\|_{e,p\to q}$ 
is comparable to $\|F_\varphi\|_s^{1/q}$ with $s=p/(p-q)$.

\begin{theorem}\label{thm:compointro}
 Let $1\leq q<p$, let $\varphi:\mathbb D\to\mathbb D$ be holomorphic with $\sigma(E_\varphi)>0$. Set $s=p/(p-q)$. Then
 $$\|F_\varphi\|_{L^s}^{1/q}\leq \|C_\varphi\|_{e,p\to q}\leq 2\|F_\varphi\|_{L^s}^{1/q}.$$
 Moreover, when $q=2,$ $\|C_\varphi\|_{e,p\to 2}=\|F_\varphi\|_s^{1/2}.$
\end{theorem}

The proof of this theorem will be given in Section \ref{sec:CO} in the wider context of composition operators 
on the Hardy spaces of the complex unit ball $\mathbb B_d$. It will use general results on inclusion 
operators inspired by \cite{BJ05,Pau16} which will be developed in Section \ref{sec:IO}.

\subsection{Multipliers on Hardy spaces of Dirichlet series}
We turn to our second example, multipliers on Hardy spaces of Dirichlet series. 
The Hardy spaces of Dirichlet series $\mathcal H^p$ were introduced by Hedenmalm, Lindqvist and Seip \cite{HLS}
for $p=2$ and by the author \cite{BAYMONAT} for the remaining cases in the range $p\in[1,+\infty]$. 
A way to define these spaces is to consider first the following norm in the space of
Dirichlet polynomials (i.e. all finite sums $\sum_{n=1}^N a_n n^{-s},\ a_n\in\mathbb C$, $N\in\mathbb N$):
$$\left\|\sum_{n=1}^N a_nn^{-s}\right\|_p^p=\lim_{T\to+\infty}\frac1{2T}\int_{-T}^T \left|\sum_{n=1}^N a_n n^{it}\right|^pdt.$$
The space $\mathcal H^p,$ $1\leq p<+\infty,$ is then defined as the completion of the Dirichlet polynomials
under this norm. Functions in $\mathcal H^p$ are Dirichlet series which converge in the half-plane $\mathbb C_{1/2}$ and are holomorphic there,
where for $a>0,$ $\mathbb C_a=\{s\in\CC:\ \Re e(s)>a\},$ 
We also need to introduce $\mathcal H^\infty$, the space of Dirichlet series that define a bounded holomorphic
function on the half-plane $\mathbb C_0$. It is endowed with the norm $\|D\|_\infty=\sup_{\Re e(s)>0}|D(s)|$.

The multipliers of $\mathcal H^p$ have been characterized in \cite{HLS,BAYMONAT}. A holomorphic self-map $D:\CC_{1/2}\to\CC_{1/2}$ induces a bounded map
$M_D:\mathcal H^p\to\mathcal H^p,$ $f\mapsto Df$ if and only if $D\in\mathcal H^\infty.$ In that case, $\|M_D\|_{p\to p}=\|D\|_\infty.$
Very recently, the multipliers between different Hardy spaces have been studied in \cite{FGS23}. In that paper, it is shown that
\begin{itemize}
 \item there is no bounded multiplier from $\mathcal H^p$ into $\mathcal H^q$ if $1\leq p<q<+\infty$;
 \item for $1\leq q<p<+\infty,$ $D$ induces a bounded map from $\mathcal H^p$ into $\mathcal H^q$ if and only if $D\in\mathcal H^r,$
 with $r=pq/(p-q)$. In that case, $\|M_D\|_{p\to q}=\|D\|_r$ and
 $$\|D\|_q\leq \|M_D\|_{e,p\to q}\leq\|D\|_r;$$
 \item for $p>1,$ $\|M_D\|_{e,p\to p}=\|D\|_\infty$; for $p=1,$
 $$\frac 12\|D\|_\infty\leq \|M_D\|_{e,1\to 1}\leq\|D\|_\infty.$$
\end{itemize}
We fully complete the picture by computing the essential norm in the remaining cases:
\begin{theorem}\label{thm:dirichletseries}
 \begin{enumerate}[(a)]
  \item Let $1\leq q<p$ and $D\in\mathcal H^r$ with $r=pq/(p-q)$. Then $\|M_D\|_{e,p\to q}=\|D\|_r.$
  \item Let $D\in\mathcal H^\infty$. Then $\|M_D\|_{e,1\to 1}=\|D\|_\infty.$
 \end{enumerate}
\end{theorem}
Our method of proof will use the Bohr lift and harmonic analysis on the polytorus. As a consequence, we will get results
corresponding to Theorem \ref{thm:dirichletseries} for multipliers on Hardy spaces of the polytorus which seem new even for the circle
(see \cite{SZ03,FGS23} for details in this case).

\subsection{Multipliers on Lebesgue spaces}
Our final example deals with multipliers on Lebesgue spaces without any extra structure. 
Let $(\Omega,\mathcal A,\mu)$ be a $\sigma$-finite measure space and let $u:\Omega\to\Omega$ be measurable. 
It is only recently that the essential norm of the multiplier $M_u:f\mapsto uf$, as an operator on $L^p(\mu),$
$p\geq 1,$ has been computed (see \cite{CLR21,Voi22}). We shall do the same when $M_u$ is viewed as an operator
from $L^p(\mu)$ to $L^q(\mu)$ with $1\leq q<p$ (continuity has been characterized in \cite{TaYo99} and is
equivalent to $u\in L^r(\mu)$, $r=pq/(p-q)$ and compactness has been characterized
in \cite{LoLoh20} in the more general context of weighted composition operators).
 In order to describe that result, we recall that the measure space can be
decomposed as a disjoint union $\Omega=\Omega_d\cup\Omega_a,$ where $\Omega_d,\Omega_a\in\mathcal A,$
the restriction $\mu_d$ of $\mu$ to $\Omega_d$ is a diffuse measure and the restriction $\mu_a$ of
$\mu$ to $\Omega_a$ is purely atomic. Namely,
\begin{itemize}
 \item for any measurable subset $A$ of $\Omega_d$ with $\mu_d(A)>0,$ for every $\alpha\in (0,\mu_d(A)),$
 there exists $A'\in\mathcal A$ with $A'\subset A$ and $\mu_d(A')=\alpha.$
 \item $\Omega_a$ is the disjoint union of a sequence $(A_n)$ of atoms (any measurable
 subset of $A_n$ has measure equal to $0$ or $\mu_a(A_n)$).
\end{itemize}
We shall also recall that $(\Omega,\mathcal A,\mu)$ is a separable measure space provided there
exists a sequence $(B_n)\subset\mathcal A$ such that, for any $B\in\mathcal A,$ for any $\veps>0,$
one may find $n\geq 1$ such that $\mu(B\Delta B_n)<\veps.$ Under this assumption,
for any $p\in[1,+\infty),$ $L^p(\mu)$ is separable:  the set of steps functions $\mathbf 1_{B_n}$
spans a dense subspace of $L^p(\mu).$

\begin{theorem}\label{thm:multiplierlebesgue}
 Let $1\leq q<p$ and set $r=pq/(p-q)$. Let $(\Omega,\mathcal A,\mu)$ be a $\sigma$-finite separable
 measure space and let $u\in L^r(\mu)$. Then $\|M_u\|_{e,p\to q}=\|u_{|\Omega_d}\|_r.$
\end{theorem}

If we allow $p=+\infty$, we lose a factor $1/2$ in the estimate of the essential norm.
\begin{theorem}\label{thm:multiplierlebesgueinfty}
Let $q\geq 1,$ let $(\Omega,\mathcal A,\mu)$ be a $\sigma$-finite separable
 measure space and let $u\in L^q(\mu)$. Then $\frac 12 \|u_{|\Omega_d}\|_q\leq \|M_u\|_{e,\infty\to q}\leq \|u_{|\Omega_d}\|_q.$
\end{theorem}

\subsection{A general argument}
We shall use several times the following lemma, inspired by \cite[Proposition 2.3]{Cha12}. 
\begin{lemma}\label{lem:generalargument}
Let $X,Y$ be Banach spaces, let $T\in\mathcal L(X,Y)$ and let $\lambda>0.$ 
\begin{enumerate}[a)]
\item Let $(\mathcal R_n)\subset \mathcal L(Y)$ be a sequence of bounded 
operators such that $\|\mathcal R_n\|\leq \lambda$ for all $n.$ 
Assume that $(\mathcal R_n)$ converges pointwise to $0.$ Then 
$$\|T\|_{e,X\to Y}\geq \frac 1\lambda \limsup_n \|\mathcal R_n T\|_{X\to Y}.$$
\item Let $(\mathcal Q_n)\subset \mathcal L(X)$ be a sequence of compact operators
and let $\mathcal R_n=\mathrm{Id}_X-\mathcal Q_n.$ Then $\|T\|_{e,X\to Y}\leq \liminf_n \| T\mathcal R_n\|_{X\to Y}.$
\end{enumerate}
\end{lemma}
\begin{proof}
\begin{enumerate}[a)]
\item Let $K:X\to Y$ be  compact. Then 
\begin{align*}
\|T-K\|_{X\to Y}&\geq \frac 1\lambda \|\mathcal R_n T-\mathcal R_n K\|_{X\to Y}\\
&\geq \frac 1\lambda\|\mathcal R_n T\|_{X\to Y}-\frac 1{\lambda}\|\mathcal R_n K\|_{X\to Y}.
\end{align*}
Now, since $K$ is compact, $(\|\mathcal R_n\|)$ is bounded and $(\mathcal R_n)$ converges pointwise to $0$; it follows from a standard compactness argument that 
$\|\mathcal R_n K\|_{X\to Y}$ tends to $0.$ Hence
$$\|T-K\|_{X\to Y}\geq\frac 1{\lambda} \limsup_n \|\mathcal R_n T\|_{X\to Y}$$
and we conclude by taking the infimum over the compact operators $K:X\to Y.$
\item This is an easy consequence of 
$$\|T\|_{e,X\to Y}=\|T\mathcal R_n+T\mathcal Q_n\|_{e,X\to Y}\leq \|T\mathcal R_n\|_{X\to Y}.$$
\end{enumerate}
\end{proof}

\subsection*{Notation}
Throughout this paper, we shall use the following notations. For $p\geq 1,$ $p^*$ will
stand for the conjugate exponent of $p.$ We shall denote by $\sigma$  the rotation invariant probability measure on $\Sd$. For two points $z,w\in \CC^d,$ we write
$$\langle z,w\rangle=\sum_{j=1}^d z_j\overline{w_j}$$
and $|z|=\sqrt{\langle z,z\rangle}.$ If we consider two functions $f:E\to \RR,$ 
we write $f\lesssim g$ is there is some $c>0$ such that $f\leq cg$ and $f\asymp g$ if 
$f\lesssim g$ and $g\lesssim f.$

\section{Inclusion operators}\label{sec:IO}

\subsection{Some results on functions on the ball}

Let $\varphi:\Bd\to\Bd$ be holomorphic. For almost every $\xi\in \Sd$, $\varphi^*(\xi)=\lim_{r\to 1}\varphi(r\xi)$
exists. Thus we may regard $\varphi$ as a map of $\overline{\Bd}$ into $\overline{\Bd}$ and we will usually
continue to write $\varphi$ for this map, and reserve the notation $\varphi^*$ for the map from $\Sd$ into $\overline{\Bd}$
as defined above.

The existence of inner functions on $\Bd$ will play an essential role. In particular, we shall use the following corollary (see \cite{Low82}): for every $G:\Sd\to (0,+\infty)$ continuous,
one may find $f\in H^\infty(\Bd)$ such that $|f|=G$ a.e. on $\Sd.$
In particular this yields the following lemma.
\begin{lemma}\label{lem:superinner}
 Let $1\leq q<p$ and set $s=p/(p-q).$ Then for all $F:\mathbb S^d\to[0,+\infty)$ measurable, 
 $$\|F\|_{L^s(\sigma)}=\sup\left(\int_{\Sd} F|g|^qd\sigma:\ g\in B_{H^p(\Bd)}\right).$$
\end{lemma}
\begin{proof}
 This follows from
 $$\|F\|_s=\sup\left(\int_{\Sd} F Gd\sigma:\ G:\Sd\to(0,+\infty)\textrm{ continuous},\ \|G\|_{s^*}=1\right\}$$
 and from $\|g\|_p=1$ provided $g\in H^\infty(\Bd)$ is such that $|g|=G^{1/q}$ a.e. on $\Sd$ with $\|G\|_{s^*}=1$ (here, $s^*=p/q$).
\end{proof}

For $\xi\in\Sd,$ the admissible approach region $\Gamma(\xi)$ is defined as 
$$\Gamma(\xi)=\left\{z\in\Bd: |1-\langle z,\xi\rangle|<1-|z|^2\right\}.$$
As a consequence of Fubini's theorem, one can prove (see e.g. \cite[equation (2.1)]{Pau16}) that for all nonnegative mesurable functions $f$ and for all positive
Borel measures $\mu,$
\begin{equation}\label{eq:pau2}
\int_{\Bd}f(z) d\mu(z)\asymp \int_{\Sd}\int_{\Gamma(\xi)} f(z) \frac{d\mu(z)}{(1-|z|^2)^d}d\sigma(\xi).
\end{equation}

If $f$ is a function on $\Bd,$ its maximal function $\mathcal Mf$ is defined on $\Sd$ by $\mathcal Mf(\xi)=\sup_{z\in \Gamma(\xi)}|f(z)|$. The maximal function has the following $L^p$-boundedness property (\cite[Theorem 5.4.10]{rudinball}): for all $p>1,$ there exists $A(p)>0$ such that, for all $f\in H^p(\Bd)$,  
$\|\mathcal Mf\|_{L^p(\sigma)}\leq A(p)\|f\|_{H^p}$.

\subsection{Essential norms of inclusion operators on $H^p(\Bd)$}

Let $\mu$ be a positive Borel measure on $\Bd$. We are intested in the inclusion operator $J_\mu:H^p(\Bd)\to L^q(\mu)$ when $p>q\geq 1.$ This operator
has already been investigated in \cite{Pau16} where it is shown that $J_\mu$ is continuous if and only if $\hat\mu:\xi\in\Sd\mapsto \int_{\Gamma(\xi)} \frac{d\mu(z)}{(1-|z|^2)^d}\in L^s(\sigma)$ where $s=p/(p-q).$

Our first result is that the continuity of $J_\mu$ implies its compactness.

\begin{theorem}\label{thm:compactinclusion}
Let $p>q\geq 1,$ let $\mu$ be a positive Borel measure on $\Bd$ and let $s=p/(p-q)$. The following assertions are equivalent:
\begin{enumerate}[(i)]
\item $J_\mu:H^p(\Bd)\to L^q(\mu)$ is continuous.
\item $J_\mu:H^p(\Bd)\to L^q(\mu)$ is compact.
\item $\displaystyle \hat\mu: \xi\in\Sd\mapsto \int_{\Gamma(\xi)}\frac{d\mu(z)}{(1-|z|^2)^d}$ belongs to $L^s(\sigma).$
\end{enumerate}
\end{theorem}
\begin{proof}
Only the implication $(iii)\implies (ii)$ has to be done. Since $H^p(\Bd)$ is reflexive (recall that $p>1$) we only have to show that $J_\mu$ is completely continuous. Let $(f_n)$ be a weakly-null sequence in $H^p(\Bd)$. Using \eqref{eq:pau2}, we have to prove that 
\begin{equation}\label{eq:compactinclusion}
\int_{\Sd}\int_{\Gamma(\xi)}|f_n(z)|^q \frac{d\mu(z)}{(1-|z|^2)^d}d\sigma(\xi)\to 0.
\end{equation}
Let $\veps>0$, let $r\in(0,1)$ and let us set $\Gamma_r(\xi)=\{z\in \Gamma(\xi):\ |z|\leq r\}$. On the one hand, 
\begin{align}
\int_{\Sd}\int_{\Gamma(\xi)\backslash\Gamma_r(\xi)}|f_n(z)|^q \frac{d\mu(z)}{(1-|z|^2)^d}d\sigma(\xi)&\leq \int_{\Sd} |\mathcal Mf_n(\xi)|^q  \int_{\Gamma(\xi)\backslash\Gamma_r(\xi)}\frac{d\mu(z)}{(1-|z|^2)^d}d\sigma(\xi) \nonumber \\
&\lesssim \|f_n\|_p^q \left(\int_{\Sd}|\widehat{\mu_r}(\xi)|^sd\sigma(\xi)\right)^{1/s}
\label{eq:pau}
\end{align}
where we have used Hölder's inequality with exponents $s$ and $s^*=p/q$ and we have set $\mu_r$ the restriction of $\mu$ to $\Gamma(\xi)\backslash \Gamma_r(\xi)$.
Observe that, for all $r\in(0,1)$ and all $\xi\in\Sd,$ $\widehat{\mu_r}(\xi)\leq \widehat{\mu}(\xi)$ and $\hat \mu\in L^s(\sigma).$ We prove that $\widehat{\mu_r}(\xi)$ 
tends to $0$ as $r$ tends to $1$ for almost all $\xi\in\Sd.$ Write 
$$\widehat{\mu_r}(\xi)=\int_{\Gamma(\xi)}F_r(z)d\mu(z)$$
with $F_r(z)=\mathbf 1_{\overline{\Bd}\backslash r\overline{\Bd}}(z)\frac{1}{(1-|z|^2)^d},$ $z\in\Gamma(\xi).$ Let $\xi$ be such that $\widehat{\mu}(\xi)<+\infty$ (this holds for a.e. $\xi$ since $\widehat{\mu}\in L^s(\sigma)$). Then
$1/(1-|z|^2)^d\in L^1(\Gamma(\xi),\mu)$ and $F_r(z)\to 0$ as $r\to 1$. 
Lebesgue's theorem implies that $\widehat{\mu_r}(\xi)$ tends to zero.  Therefore,
by a second application of Lebesgue's theorem, $\|\widehat{\mu_r}\|_s$ tends to zero and it follows, since $(\|f_n\|_p)$ is bounded, that for $r$ sufficiently close to $1$ and for all $n\geq 1,$
$$\int_{\Sd}\int_{\Gamma(\xi)\backslash\Gamma_r(\xi)}|f_n(z)|^q \frac{d\mu(z)}{(1-|z|^2)^d}d\sigma(\xi)\leq\veps.$$
Such a value of $r$ being fixed, we now observe that $(f_n)$ converges uniformly to $0$ in $r\overline{\Bd}.$ This implies easily that, for $n$ large enough,
$$\int_{\Sd}\int_{\Gamma_r(\xi)}|f_n(z)|^q \frac{d\mu(z)}{(1-|z|^2)^d}d\sigma(\xi)\leq\veps.$$
Hence \eqref{eq:compactinclusion} is proved and $J_\mu$ is completely continuous.
\end{proof}
In view of our application to composition operators, we are going to enlarge our setting 
to positive Borel measures $\mu$ on the closed ball $\overline{\Bd}$ (see \cite{BJ05} 
where the one-variable case is studied). Let $\mu$ be such a mesure, set $\mub$ its restriction to $\Bd$ and $\mus$ its restriction to $\Sd$. We are interested in the inclusion operator $J_\mu:H^p(\Bd)\to L^q(\mu),$ $1\leq q<p.$ Now observe that 
functions in $H^p(\Bd)$ are only defined almost everywhere on $\Sd$. This leads us
to restrict ourselves to the case where $\mus=Fd\sigma$ for some nonnegative $F\in L^1(\sigma).$ Under these assumptions, we can characterize the continuity of $J_\mu$: 
\begin{proposition}\label{prop:inclusionbord}
Let $p>q\geq 1,$ let $s=p/(p-q)$ and let $\mu=\mub+Fd\sigma$ be a positive Borel measure on $\overline{\Bd}$. Then $J_\mu:H^p(\Bd)\to L^q(\mu)$ is bounded if and only if 
\begin{enumerate}[a)]
\item $\displaystyle \hat\mu: \xi\in\Sd\mapsto \int_{\Gamma(\xi)}\frac{d\mu(z)}{(1-|z|^2)^d}$ belongs to $L^s(\sigma).$
\item $F$ belongs to $L^s(\sigma).$
\end{enumerate}
\end{proposition}
\begin{proof}
That a) and b) imply the continuity of $J_\mu$ follows partly from Pau's result and partly from Hölder's inequality. Indeed, for $f\in H^p(\Bd),$ 
$$\int_{\Sd}|f|^q Fd\sigma\leq \left(\int_{\Sd}|f|^p d\sigma\right)^{q/p}\left(\int_{\Sd} F^sd\sigma\right)^{1/s}$$
where we have used Hölder's inequality with the exponents $p/q$ and $s.$ Conversely assume that $J_\mu$ is bounded. Again, a) follows from Pau's result. To prove b) we observe that for all $g\in H^p(\Bd)$ with norm $1,$ 
$$\int_{\Sd}F|g|^q d\sigma\leq \|J_\mu\|_{p\to q}^q$$
and we conclude by Lemma \ref{lem:superinner}.
\end{proof}

We now characterize compactness and compute the essential norm.
\begin{theorem}\label{thm:inclusion}
Let $p>q\geq 1,$ let $s=p/(p-q)$ and let $\mu=\mub+Fd\sigma$ be a positive Borel measure on $\overline{\Bd}$ such that $J_\mu$ is continuous. Then $\|J_\mu\|_{e,p\to q}=\|F\|_s^{1/q}.$
\end{theorem}
\begin{proof}
Let $T:H^p(\Bd)\to L^q(\mu),\ f\mapsto f\mathbf 1_{\Bd}$. Then $T=i\circ J_{\mub}$ where $i$ is the inclusion operator $L^q(\mub)\to L^q(\mu)$. Since $J_{\mub}$ is compact, $T$ is also compact. Moreover, for $f\in H^p(\Bd),$ 
\begin{align*}
\|(J_\mu-T)f\|_{L^q (\mu)}&=\left(\int_{\Sd}|f|^q d\mus\right)^{1/q}\\
&=\left(\int_{\Sd}|f|^q Fd\sigma\right)^{1/q}\\
&\leq \|f\|_p \cdot \|F\|_s^{1/q}
\end{align*}
by Hölder's inequality applied to the conjugate exponents $p/q$ and $s$.

Conversely, let $\veps>0$ and let $g\in B_{H^p}$ be such that 
$$\int_{\Sd}F|g|^q d\sigma\geq \|F\|_s-\veps.$$
 Let also $I$ be an inner function on $\Bd$ with $I(0)=0$ and let us consider $g_{k}=I^k g.$ Then $(g_{k})_k$ converges weakly to $0$ (see \cite[p. 37]{GorMacCluer04}). Therefore, 
$$\|J_\mu\|_{e,p\to q}\geq \limsup_{k\to+\infty}
\frac{\|J_\mu(g_{k})\|_q}{\|g_{k}\|_p}.$$
Now, $\|g_{k}\|_p=\|g\|_p=1$ whereas 
$$\|J_\mu(g_{k})\|_q\geq\left(\int_{\Sd}|g|^q F d\sigma\right)^{1/q}\geq \left(\|F\|_s-\veps\right)^{1/q}.$$
Since $\veps>0$ is arbitrary, we get the lower inequality $\|J_\mu\|_{e,p\to q}\geq\|F\|_s^{1/q}.$
\end{proof}

\section{Composition operators}\label{sec:CO}

\subsection{Composition operators on Hardy spaces of the ball}

We turn to composition operators on the Hardy spaces of the ball $\Bd.$
We fix $\varphi$ a holomorphic self-map of $\Bd$ such that $C_\varphi$ induces a bounded composition operator on some (therefore, on all) Hardy spaces $H^p(\Bd)$. 
In particular, this implies the following facts which we shall use repeatedly:
\begin{itemize}
\item[(H1)] no set of positive measure in $\Sd$ is mapped by $\varphi$ onto a set of measure $0$ in $\Sd$ (see Corollary 3.38 of \cite{CoMcl95});
\item[(H2)] for any $f\in H^p(\Bd)$, $(f\circ\varphi)^*(\xi)=f(\varphi^*(\xi))$ for a.e. $\xi\in\Sd$ (see Lemma 1.6 of \cite{Mcl85}).
\end{itemize}
We point out that, in what follows, we could replace the assumption {\it $C_\varphi$ is continuous on $H^p$} by these two assumptions. They are satisfied for any
holomorphic self-map of $\Bd$ when $d=1.$

Suppose now that $\sigma(E)>0,$ where $E=E_\varphi=\{\xi\in\Sd:\ |\varphi(\xi)|=1\}.$ Then $\varphi^*_{|E}$ induces a 
nonsingular transformation from $(E,\sigma)$ into $(\varphi^*(E),\sigma).$ Let $F_\varphi$ be the Radon-Nikodym derivative
of $\sigma_{|E}\circ(\varphi^*)_{|\varphi^*(E)}^{-1}$ with respect to $\sigma_{|\varphi^*(E)}.$ We extend $F_\varphi$ on $\Sd$ outside $\varphi^*(E)$ by setting it equal to $0$.
If $\sigma(E)=0,$ we just set $F_\varphi=0$ on $\Sd.$ We intend to prove
the following general version of Theorem \ref{thm:compointro}.

\begin{theorem}\label{thm:compoball}
 Let $1\leq q<p$ and set $s=p/(p-q)$. For all analytic maps $\varphi:\Bd\to\Bd$ inducing a bounded operator $C_\varphi:H^p(\Bd)\to H^p(\Bd)$,
 $$\|F_\varphi\|_s^{1/q}\leq \|C_\varphi\|_{e,p\to q}\leq \min(2,\|P_q\|)\|F_\varphi\|_s^{1/q}$$
 where $P_q:L^q(\sigma)\to H^q(\Bd)$ is the Szegö projection.
\end{theorem}
In particular, for $q=2,$ we get $\|C_\varphi\|_{e,p\to q}=\|F_\varphi\|_s^{1/2}.$
\begin{proof}
 If $\sigma(E)=0$, then $C_\varphi$ is compact by \cite[Corollary 2]{GorMacCluer04} and there is nothing to prove. Therefore we will assume $\sigma(E)>0$.
 Let $\mu_\varphi=\sigma\circ(\varphi^*)^{-1}$ be the pullback measure of $\sigma$ by $\varphi^*$, which is a measure on  $\overline{\Bd}$. Its restriction to $\Sd$ is
 $F_\varphi d\sigma.$
 The change of variables formula shows that, for any $f\in H^p(\Bd)$ 
 $$\|C_\varphi(f)\|_q=\|J_{\mu_\varphi}(f)\|_{L^q(\mu_\varphi)}.$$
 However, without any extra work, this does not rely directly $\|C_\varphi\|_{e,p\to q}$ to $\|J_{\mu_\varphi}\|_{e,p\to q}.$
 We first give a proof working for $q>1$ (recall that the Szegö projection is bounded if and only if $q>1$).
 
 Let us introduce
 $$\begin{array}{rcllrcl}
    W_q:L^q(\mu_\varphi)&\to&L^q(\sigma)&\quad\quad\quad&R_q:H^q(\Bd)&\to&L^q(\sigma)\\
    g&\mapsto&g\circ\varphi^*&&f&\mapsto&f^*.
   \end{array}
   $$
The maps $W_q$ and $R_q$ are both isometries and from $(f\circ\varphi)^*=f\circ \varphi^*,$ we deduce that
$$R_q\circ C_\varphi=W_q\circ J_{\mu_\varphi}:H^p(\Bd)\to L^q(\sigma).$$
Observe also that $P_qR_q=\textrm{Id}_{H^q}$. Let 
finally $K:H^p(\Bd)\to L^q(\mu_\varphi)$ be a compact operator. Then
\begin{align*}
 \|J_{\mu_\varphi}-K\|_{p\to q}&\geq \|W_q J_{\mu_\varphi}-W_q K\|_{p\to q}\\
 &\geq \|R_q C_\varphi-W_q K\|_{p\to q}\\
 &\geq \|P_q\|^{-1}\|C_\varphi-PW_qK\|_{p\to q}\\
 &\geq \|P_q\|^{-1} \|C_\varphi\|_{e,p\to q}
\end{align*}
which shows that $\|C_\varphi\|_{e,p\to q}\leq \|P_q\|\cdot \|J_{\mu_\varphi}\|_{e,p\to q}.$ Conversely, let us define
$V_q:L^q(\sigma)\to L^q(\mu_\varphi)$ by duality: for $f\in L^q(\sigma)$ and
$g\in L^{q^*}(\mu_\varphi)$, 
$$\int_{\overline{\Bd}}V_q f \bar gd\mu_{\varphi}=\int_{\Sd}f \overline{W_{q^*}g}d\sigma.$$
In particular, $\|V_q\|\leq 1$ since $W_{q^*}$ is an isometry.
Observe also that $V_qW_q=\textrm{Id}_{L^q(\mu_\varphi)}$ since for $(f,g)\in L^q(\mu_\varphi)\times L^{q^*}(\mu_\varphi),$
\begin{align*}
 \int_{\overline{\Bd}}V_q W_q(f)\bar gd\mu_{\varphi}&=\int_{\Sd}W_q(f)\overline{W_{q^*}(g)}d\sigma\\
 &=\int_{\overline{\Bd}}f\bar gd\mu_{\varphi}
\end{align*}
by the change of variables formula. Now, let $K:H^p(\Bd)\to H^q(\Bd)$ be a compact operator. Then
\begin{align*}
 \|C_\varphi-K\|_{p\to q}&\geq \|R_q C_\varphi-R_q K\|_{p\to q}\\
 &\geq \|W_q J_{\mu_\varphi}-R_q K\|_{p\to q}\\
 &\geq \|J_{\mu_\varphi}-V_q R_q K\|_{p\to q}\\
 &\geq \|J_{\mu_\varphi}\|_{e,p\to q}
\end{align*}
which shows the reverse inequality, $\|J_{\mu_\varphi}\|_{e,p\to q}\leq \|C_\varphi\|_{e,p\to q}.$ Now we conclude because $F_{\mu_\varphi}=F_\varphi$ 
by definition.

We now prove the upper inequality for all values of $q\geq 1$ and for the constant $2$. 
Let $n\geq 1$ and let $\mathcal Q_n:H^p\to H^p,\ f\mapsto f\left(\left(1-\frac 1n\right)\cdot\right).$ Then $\mathcal Q_n$ is a compact operator
with norm less than $1$. Let $\mathcal R_n=I-\mathcal Q_n$, $\|\mathcal R_n\|\leq 2.$ Then for all $n\geq 1,$
\begin{align*}
 \|C_\varphi\|_{e,p\to q}&\leq \liminf_n  \|C_\varphi \mathcal R_n\|_{p\to q}.
\end{align*}
Let $f\in H^p(\Bd)$ with $\|f\|\leq 1$ and let $r\in (0,1).$
\begin{align*}
 \|C_\varphi \mathcal R_n(f)\|^q&=\int_{\Sd}|\mathcal R_n(f)\circ \varphi|^qd\sigma\\
 &=\int_{r\overline{\Bd}}|\mathcal R_n(f)|^qd\mu_{\varphi}+\int_{\overline{\Bd}\backslash r\overline{\Bd}}|\mathcal R_n(f)|^q d\mu_{\varphi}\\
 &=: I_{1,n}(r)+I_{2,n}(r).
\end{align*}
By Cauchy integral formula and by \cite[Theorem 4.17]{Zhu05}, for any $z\in r\overline{\mathbb B_d},$
\begin{align*}
 |\mathcal R_n(f)(z)|&\leq\frac 1n \sup_{w\in r\overline{\Bd}}|f'(w)|\\
 &\leq \frac{C(r,d)}n \|f\|_p
\end{align*}
so that $I_{1,n}(r)\leq C(\varphi,r,d)/n^q$ where $C(\varphi,r,d)$ only depends on $\varphi,$ $r$ and $d$.

Let us turn to $I_{2,n}(r)$ and let us denote by $\mu_r$ the restriction 
of $\mu_\varphi$ to $\overline{\Bd}\backslash \Bd.$ Then by combining Pau's argument (see inequality (\ref{eq:pau})) and the proof of Proposition \ref{prop:inclusionbord}, we get that
\begin{align*}
 I_{2,n}(r)&=\int_{\overline{\Bd}} |\mathcal R_n(f)|^q d\mu_r\\
 &\leq \left(\|F_{\varphi}\|_s +A(p)\|\widehat{\mu_r}\|_s\right)\|\mathcal R_n(f)\|_p^q\\
 &\leq 2^q \|F_\varphi\|_s+2^qA(p)\|\widehat \mu_r\|_s
\end{align*}
where $A(p)$ only depends on $p$.
But as in the proof of Theorem \ref{thm:compactinclusion}, we get that $\|\widehat{\mu_r}\|_s\to 0$ as $r\to 1.$
Putting everything together, we finally get $\liminf_{n\to+\infty} \|C_\varphi \mathcal R_n\|_{p\to q}\leq 2\|F_{\varphi}\|_s^{1/q},$
which achieves the proof of the upper estimate.

We conclude by providing a proof for the lower estimate. Let $M_E$ be the operator of multiplication by ${\mathbf 1}_E$ from $H^q(\Bd)$ to $L^q(\sigma)$. It is shown in 
\cite{GorMacCluer04} that 
$$\|C_\varphi\|_{e,p\to q}\geq \|M_E C_\varphi\|_{p\to q}.$$
Now, 
\begin{align*}
\|M_E C_\varphi\|_{p\to q}^q &=\sup_{g\in B_{H^p}} \int_E |g\circ\varphi|^q d\sigma\\
&=\sup_{g\in B_{H^p}} \int_{\Sd}|g|^q d\mu_{\varphi}\\
&=\sup_{g\in B_{H^p}}\int_{\Sd}|g|^q F_\varphi d\sigma\\
&=\|F_\varphi\|_s
\end{align*}
by Lemma \ref{lem:superinner}.
\end{proof}

\subsection{Weighted composition operators} \label{sec:wcoh2}
Without extra-work, we can also give an estimate of the essential norm of weighted
composition operators. Let $u:\Bd\to\CC$ and $\varphi:\Bd\to\Bd$ be holomorphic.
Then the weighted composition operator $uC_\varphi$ is defined by $(uC_\varphi)(f)=u\cdot (f\circ \varphi)$. Again, we assume that $C_\varphi$ induces a bounded 
operator on $H^p(\Bd)$.  If $\sigma(E)>0,$ then $\varphi^*_{|E}$ induces a 
nonsingular transformation from $(E,\sigma)$ into $(\varphi^*(E),\sigma).$ Let $F_{u,\varphi}$ be the Radon-Nikodym derivative
of $|u|^q \sigma_{|E}\circ(\varphi^*)^{-1}_{|\varphi^*(E)}$ with respect to $\sigma_{|\varphi^*(E)}.$

\begin{corollary}
 Let $1\leq q<p$ and set $s=p/(p-q)$. For all analytic maps $\varphi:\Bd\to\Bd$ and $u:\Bd\to\CC^d$ such that  $C_\varphi:H^p(\Bd)\to H^p(\Bd)$ and $uC_\varphi:H^p(\Bd)\to H^q(\Bd)$ are bounded,
 $$\|F_{u,\varphi}\|_s^{1/q}\leq \|uC_\varphi\|_{e,p\to q}\leq \min(2,\|P_q\|)\|F_{u,\varphi}\|_s^{1/q}$$
 where $P_q:L^q(\sigma)\to H^q(\Bd)$ is the Szegö projection.
\end{corollary}
\begin{proof}
 Let $\mu_\varphi=(|u|^q \sigma)\circ(\varphi^*)^{-1}$ be the pullback measure of $|u|^q d\sigma$ by $\varphi$, which is a measure on  $\overline{\Bd}$. 
 The change of variables formula now writes for any $f\in H^p(\Bd),$ 
 $$\|C_\varphi(f)\|_q=\|J_{\mu_\varphi}(f)\|_{L^q(\mu_\varphi)}.$$
 The proof of the upper estimate follows exactly that of Theorem \ref{thm:compoball}. For the lower estimate, we can do exactly the same proof provided we show that 
 $$\|uC_\varphi\|_{e,p\to q}\geq \|(\mathbf 1_E u)C_\varphi\|_{p\to q}.$$
 Let $K$ be a compact operator and let $I$ be an inner function on $\Bd$ such that $I(0)=0$. Then for any $f$ in the unit ball of $H^p(\Bd),$
 $$\|uC_\varphi-K\|_{p\to q}\geq \|uC_\varphi(I^n f)\|_q-\|K(I^n f)\|_q.$$
 Since $(I^n f)$ goes weakly to zero, and since 
 \begin{align*}
 \|uC_\varphi(I^n f)\|^q_q&=\int_{\Sd}|u(\xi)|^q |I^n\circ\varphi^*|^q |f\circ\varphi^*|^q d\sigma\\
 &\to \int_ E |u(\xi)|^q |f\circ\varphi^*|^q d\sigma,
 \end{align*}
 we get the result. Observe that the last part of the proof uses (H1) and (H2) 
 to ensure the a.e. convergence of $|I^n\circ\varphi^*|$ to $\mathbf 1_{E}.$
\end{proof}
The most important case in the previous theorem happens when $\varphi(z)=z$. Then $uC_\varphi$ is the multiplication operator $M_u.$ In the setting, one can say more, 
since $\|M_u\|_{e,p\to q}\leq \|M_u\|_{p\to q}=\|u\|_r$ where $r=pq/(p-q).$ 
Since moreover $F_{u,\varphi}=|u|^q,$ we get $\|M_u\|_{e,p\to q}=\|u\|_r.$ The extension of this result to Hardy 
spaces of the polydisc and to Dirichlet series will be the subject of the next section.

\section{Multipliers on spaces of Dirichlet series}

\subsection{Some facts on Hardy spaces of Dirichlet series}
We shall need the following facts on Dirichlet series. We refer to \cite{FGS23} and the references therein for details.
Let $N\geq 1$ and let $D(s)=\sum_{n\geq 1}a_n n^{-s}$ be a Dirichlet series. We denote by $D_N$ the restriction
of $D$ to the first $N$ prime numbers: $D_N(s)=\sum_{\textrm{gpd}(n)\leq p_N}a_nn^{-s}$ where $\textrm{gpd}(n)$
denotes the biggest prime divisor of $n$ and $(p_n)_{n\geq 1}$ is the increasing family of prime numbers. 
Then the map $\mathcal P_N:\mathcal H^p\to\mathcal H^p,$ $D\mapsto D_N$ is a contraction for any $p\in[1,+\infty]$
and when $p\in[1,+\infty),$ $\mathcal P_N(f)\to f$ in $\mathcal H^p$ as $N\to+\infty.$ If $p=+\infty,$ the convergence holds in the weak-star topology and it is still true that 
$\|\mathcal P_Nf\|_\infty\to \|f\|_\infty$ as $N\to+\infty$ (see \cite[Chapter 5]{Defantbook}).
In the following, we will set $\mathcal H^p_N=\mathcal P_N(\mathcal H^p)$.

Hardy spaces of the infinite polytorus and of Dirichlet series are linked by the Bohr point of view. Let $f(s)=\sum_{n=1}^N a_n n^{-s}$
be a Dirichlet polynomial. Any integer $n$ factorizes as $n=p_1^{\alpha_1}\cdots p_r^{\alpha_r}$. We define the Bohr lift of $f$ by
$$\mathcal L(f)=\sum_{n=1}^N a_n z^{\alpha(n)}$$
where $\alpha(n)=(\alpha_1,\dots,\alpha_r,0,\dots)$ provided $n=p_1^{\alpha_1}\cdots p_r^{\alpha_r}.$
Then $\mathcal L$ induces an isometric isomorphism between $\mathcal H^p$ and $H^p(\TT^\infty)$ for all $p\in[1,+\infty].$
Its inverse will be denoted by $\mathcal B$ and will be called the Bohr transform. Observe that $\mathcal L$ induces 
an isometric isomorphism between $\mathcal H^p_N$ and $H^p(\TT^N)$.

\subsection{Essential norm of multipliers}
This subsection is devoted to the proof of Theorem \ref{thm:dirichletseries}.
\begin{proof}[Proof of Theorem \ref{thm:dirichletseries}, part (a)]
 By \cite[Theorem 9]{FGS23} we only need to prove the lower bound. Let $K:\mathcal H^p\to\mathcal H^q$ be a compact operator, 
 let $N\geq 1$ and let $\mathcal P_N:\mathcal H^p\to\mathcal H^p_N$ be the canonical projection. We set $D_N=\mathcal P_N(D)$
 and $K_N=\mathcal P_N K\mathcal P_N$. Then $D_N$ induces a multiplier $M_{D_N}:\mathcal H^p_N\to \mathcal H^q_N,$
 $K_N$ is a compact operator from $\mathcal H^p_N$ to $\mathcal H^q_N$ and
 \begin{align*}
  \|M_D-K\|_{p\to q}&\geq \|\mathcal P_NM_D\mathcal P_N-\mathcal P_N K\mathcal P_N\|_{p\to q}\\
  &=\|M_{D_N}-K_N\|_{p\to q}.
 \end{align*}
We move to the polydisc $\TT^N$ by considering $F_N=\mathcal L(D_N)$ and we still denote $K_N=\mathcal L\circ K_N\circ \mathcal B$. 
We intend to show that 
$$\|M_{F_N}-K_N\|_{p\to q}\geq \|F_N\|_{H^r(\TT^N)}=\|D_N\|_r.$$
Letting $N$ to $+\infty$ will yield the result, since $\|D_N\|_r\to \|D\|_r.$

We set $t=q/(p-q)$ and $G=|F_N|^t.$ Then $G\in L^p(\TT^N)$ and $\|G\|_p^p=\|F_N\|_r^r.$
There exists a sequence of trigonometric polynomials $(Q_n)$ such that $\|Q_n-G\|_p\to 0$
and $Q_n\to G$ a.e. on $\TT^N.$
For a fixed $n\geq 1,$ let $P_n=\prod_{j=1}^N z_j^d$ where $d\geq 0$ is sufficiently large so that
$P_nQ_n\in H^p(\TT^N)$ and let for $k\geq 1$ $E_{k,n}=z_1^k P_nQ_n$. Then $E_{k,n}$ belongs to $H^p(\TT^N)$,
$|E_{k,n}|=|Q_n|$ on $\TT^N$ and $E_{k,n}(z)\to 0$ as $k\to+\infty$ for any $z\in \DD^N.$
Therefore, by \cite[Lemma 13]{FGS23}, $(E_{k,n})_k$ converges to $0$ in the weak-star topology
of $H^p(\TT^N)$, therefore in its weak topology since $H^p$ is reflexive.
Now,
\begin{align*}
 \|M_{F_N}(E_{k,n})\|_q&=\left(\int_{\TT^N} |Q_n|^q |F_N|^q\right)^{1/q}\\
 &\geq \left(\int_{\TT^N} |G|^q |F_N|^q\right)^{1/q}-\left(\int_{\TT^N} |Q_n-G|^q |F_N|^q\right)^{1/q}.
\end{align*}
Therefore,
\begin{align*}
 \|M_{F_N}-K_N\|_{p\to q}&\geq \limsup_{k\to+\infty} \frac{\|M_{F_N}(E_{k,n})-K_N(E_{k,n})\|_q}{\|E_{k,n}\|_p}\\
 &\geq \frac 1{\|Q_n\|_p}\left(\left(\int_{\TT^N} |G|^q |F_N|^q\right)^{1/q}-\left(\int_{\TT^N} |Q_n-G|^q |F_N|^q\right)^{1/q}\right)\\
 &\geq \frac 1{\|Q_n\|_p}\left( \|F_N\|_r^{r/q}-\|Q_n-G\|_p^{1/q}\|F_N\|_r^{1/q}\right)
\end{align*}
by Hölder's inequality applied to the pair of conjugated exponents $p/q$ and $p/(p-q)$.
We let $n$ to $+\infty$ to get 
$$\|M_{F_N}-K_N\|_{p\to q}\geq \frac{\|F_N\|_r^{r/q}}{\|F_N\|_r^t}=\|F_N\|_r.$$
\end{proof}

\begin{remark}
Observe that the above proof is based on two arguments similar to  those introduced in the previous sections: we use that we can compute the norm of an element in $L^r(\TT^N)$
using only functions in $B_{H^p}$ and we use the existence of inner functions
on the polydisc to get a sequence going weakly to zero with prescribed modulus at the
distinguished boundary. Here Fourier analysis arguments simplify the proofs. 
\end{remark}

\begin{proof}[Proof of Theorem \ref{thm:dirichletseries}, part (b)]
 Arguing as above, it is sufficient to prove that, for each $N\geq 1,$ for each $F\in H^\infty(\TT^N)$, $F\neq 0,$ $\|M_F\|_{e,1\to 1}\geq \|F\|_\infty.$
 The main difficulty we are facing is that $\mathcal H^1$ is no longer reflexive and it is more difficult to exhibit sequences converging
 weakly to $0$. Our strategy (inspired by \cite{Voi22}) will be, given $\veps>0$, to construct a bounded sequence $(R_n)$
 in $H^1(\TT^N)$ so that, for all $m>n,$ $\int_{\TT^N} |F|\cdot |R_n-R_m|\geq (\|F\|_\infty-\veps)\|R_n-R_m\|_1.$
 This construction will be achieved by regularizing functions peaking around $\{z\in\TT^N:\ |F(z)|\geq \|F\|_\infty-\veps\}.$
 
 Thus let $\veps>0$, $\veps<\min(1/4,\|F\|_\infty)$ and let us denote by $\mu$ the Haar measure on $\TT^N$.
 There exists a decreasing sequence of measurable subsets $(A_n)$ of $\TT^N$ such that 
 $$\left\{
 \begin{array}{ll}
  |F(x)|\geq \|F\|_\infty-\veps&\textrm{for all }x\in A_n\\
  \mu(A_{n+1})\leq \frac 14\mu(A_n).
 \end{array}
 \right.$$
 If we take the convolution product of the nonnegative functions $\frac 1{\mu(A_n)}\mathbf 1_{A_n}$ with the Féjer kernel,
 we get for each $n\geq 1$ a sequence of trigonometric polynomials $(G_{n,k})_k$ such that 
 $$G_{n,k}\xrightarrow{k\to+\infty} \frac1{\mu(A_n)}\mathbf 1_{A_n}\ \textrm{a.e.}$$
 $$\forall n,k\geq 1,\ 0\leq G_{n,k}\leq \frac1{\mu(A_n)}$$
 $$\forall n,k\geq 1,\ \|G_{n,k}\|_1\leq 1.$$
Using Egorov's theorem, we obtain for each $n\geq 1$ a trigonometric polynomial $Q_n$ and a measurable set $B_n\subset\TT^N$ such that 
$$\mu(\TT^N\backslash B_n)\leq \veps\mu(A_n)$$
$$\left|Q_n-\frac1{\mu(A_n)}\mathbf 1_{A_n}\right|\leq \veps\textrm{ on }B_n$$
$$0\leq Q_n\leq \frac1{\mu(A_n)}.$$
$$\|Q_n\|_1\leq 1$$
We then multiply $Q_n$ by some unimodular polynomial $P_n=\prod_{j=1}^N z_j^d$ to get a holomorphic polynomial $R_n$ with the same modulus as $Q_n$. 
We claim that the following fact is true.

\noindent{\bf Fact.} For any $m>n\geq 1,$ 
$$\int_{\TT^N\backslash A_n}|R_n-R_m|<4\veps\textrm{ and }\int_{A_n}|R_n-R_m|\geq \frac 18.$$

Let us admit the fact for a while to achieve the proof of Theorem \ref{thm:dirichletseries}. The sequence $(R_n)$ is a bounded
sequence of $H^1(\TT^N)$. Let $K:H^1(\TT^N)\to  H^1(\TT^N)$ be compact. Extracting if necessary, we may assume that $(K(R_n))$ converges.
Let $m>n$ be such that $\|KR_m-KR_n\|\leq\veps.$
Then 
\begin{align*}
 \|(M_F-K)(R_n-R_m)\|_1&\geq \|M_F(R_n-R_m)\|_1-\veps\\
 &\geq \int_{A_n}|F|\cdot |R_n-R_m|-\veps\\
 &\geq (\|F\|_\infty-\veps)\int_{A_n}|R_n-R_m|-\veps.
\end{align*}
By the fact,
$$\int_{\TT^N}|R_n-R_m|\leq \int_{A_n}|R_n-R_m|+\int_{\TT^N\backslash A_n}|R_n-R_n|\leq (1+32\veps)\int_{A_n}|R_n-R_m|$$
so that
$$\|(M_F-K)(R_n-R_m)\|_1\geq \frac{\|F\|_\infty-\veps}{1+32\veps}\|R_n-R_m\|_1-8\veps\|R_n-R_m\|_1.$$
Since $\veps>0$ is arbitrary, we get $\|M_F\|_{e,1\to 1}\geq \|F\|_\infty.$

\medskip

It remains to prove the fact. We first observe that 
$$\int_{A_n} |R_n-R_m|\geq \int_{A_n\cap B_n\cap B_m\backslash A_m}|P_nQ_n-P_mQ_m|.$$
Now, provided $z\in A_n\cap B_n\cap B_m\backslash A_m,$
$$|P_nQ_n(z)|\geq\frac1{\mu(A_n)}-\veps\textrm{ and }|P_mQ_m(z)|\leq \veps$$
so that 
\begin{align*}
 \int_{A_n}|R_n-R_m|&\geq \mu(A_n\cap B_n\cap B_m\backslash A_m)\left(\frac 1{\mu(A_n)}-2\veps\right)\\
 &\geq \big( \mu(A_n\backslash A_m)-\mu(\TT^N\backslash B_n)-\mu(\TT^N\backslash B_m)\big) \left(\frac 1{\mu(A_n)}-2\veps\right)\\
 &\geq \frac 14\mu(A_n)\cdot\left(\frac1{\mu(A_n)}-2\veps\right)\\
 &\geq \frac 14-\frac{\veps\mu(A_n)}2\geq \frac 18 
\end{align*}
since $\veps<1/4.$ Furthermore,
\begin{align*}
\int_{\TT^N\backslash A_n}|R_n-R_m|&\leq \int_{\TT^N\backslash A_n}|R_n|+\int_{\TT^N \backslash A_n}|R_m|\\
&\leq \int_{\TT^N\backslash A_n}|R_n|+\int_{\TT^N \backslash A_m}|R_m|.
\end{align*}
We just need to study
\begin{align*}
 \int_{\TT^N\backslash A_n}|R_n|&\leq \int_{(\TT^N\cap B_n)\backslash A_n}|R_n|+\int_{\TT^N \backslash B_n}|R_n|\\
 &\leq \veps+\mu(\TT^N\backslash B_n)\times\frac1{\mu(A_n)}\\
 &\leq 2\veps.
\end{align*}
\end{proof}
A corollary of our proof is the following result.
\begin{corollary}
 Let $N\in\mathbb N\cup\{+\infty\}$, let $1\leq q<p$ and let $u\in H^r(\TT^N)$ with $r=pq/(p-q)$. Then $\|M_u\|_{e,p\to q}=\|u\|_r.$
\end{corollary}

It remains one case studied in \cite{FGS23} where we are not able to give a formula
for the essential norm: for $q\geq 1$ and $D\in\mathcal H^q,$ it is shown in \cite{FGS23} that 
$$\frac 12\|D\|_q\leq \|M_D\|_{e,\infty\to q}\leq \|D\|_q.$$
We can at least improve this for $q=2.$
\begin{proposition}\label{prop:H2}
Let $D\in\mathcal H^2.$ Then $\|M_D\|_{e,\infty\to 2}=\|D\|_2.$
\end{proposition}
\begin{proof}
Let $\mathcal Q_N$ be the orthogonal projection of $\mathcal H^2$ onto $\textrm{span}(1,2^{-2},\cdots,N^{-s})$ and let $\mathcal R_N=\textrm{Id}-\mathcal Q_N$ which has norm $1.$ By Lemma \ref{lem:generalargument}, $\|M_D\|_{e,\infty\to 2}\geq \limsup_{N\to+\infty}\|\mathcal R_N M_D\|.$ Now, let us fix $N\geq 1$
and $n\geq 1$ such that $2^n>N.$ Then
\begin{align*}
\|\mathcal R_N M_D\|_{\infty\to 2}&\geq \|\mathcal R_N M_D(2^{-ns})\|_2 \\
&\geq \|M_D(2^{-ns})\|_2-\|\mathcal Q_N M_D(2^{-ns})\|_2\\
&\geq \|D\|_2
\end{align*}
since $\mathcal Q_N M_D(2^{-ns})=0.$
\end{proof}

\subsection{Spectrum of multipliers}
We end up this section by improving a result of \cite{FGS23} regarding the spectrum of multipliers.
\begin{theorem}\label{thm:spectrum}
 Let $D\in\mathcal H^\infty$ be a non zero Dirichlet series with associated multiplication operator $M_D\in\mathcal L(\mathcal H^p),$
 $p\in[1,+\infty)$. Then $\sigma_c(M_D)\subset \overline{D(\CC_0)}\backslash D(\CC_0)$.
\end{theorem}
Here, $\sigma_c(M_d)$ denotes the continuous spectrum of $M_D$, namely the set of complex numbers $\lambda$ such that
$M_D-\lambda$ is injective and has dense but not closed range. In \cite{FGS23}, it was only shown that $\sigma_c(M_D)\subset \overline{D(\CC_0)}\backslash D(\CC_{1/2})$.
\begin{proof}
 Since $M_D-\lambda=M_{D-\lambda},$ it is sufficient to show that if $M_D$ has dense range, then $D$ does not vanish on $\CC_0$ 
 (it is easy to show that $\sigma(M_D)\subset \overline{D(\CC_0)}$, see \cite{FGS23} for details). Let $N\geq 1.$ If $M_D$
 has dense range, then $M_{D_N}:\mathcal H^p_N\to \mathcal H^p_N$ has dense range too. Assume that $D_N(s_0)=0$ for some $s_0\in\CC.$
 Since pointwise evaluation at $s_0\in\CC_0$ is continuous on $\mathcal H^p_N,$ $M_{D_N}(\mathcal H^p_N)\subset \{E\in\mathcal H^p_N:\ E(s_0)=0\}$
 cannot be dense, a contradiction.
 
 Therefore, for all $N\geq 1,$ $D_N$ do not vanish on $\CC_0$. Now, $\|D_N\|_\infty\leq \|D\|_\infty$
 and by Montel's theorem in $\mathcal H^\infty,$ upon taking a subsequence, there exists $\tilde D\in\mathcal H^\infty$
 such that $(D_{N_j})$ converges uniformly to $\tilde D$ on any half-plane $\CC_\sigma,$
 for all $\sigma>0.$ Now, since the Dirichlet series $D$ converges absolutely in $\CC_{1/2},$  $(D_{N_j}(s))$ converges
 to $D(s)$ for any $s\in\CC_{1/2}$. Hence $D=\tilde D$ on $\CC_{1/2},$ therefore on $\CC_0.$ We can now use Hurwitz theorem
 to conclude that $D$ does not vanish on $\CC_0.$
\end{proof}

\section{Multipliers on Lebesgue spaces}
\subsection{The case $p\neq+\infty$}
In this subsection we intend to prove Theorem \ref{thm:multiplierlebesgue}. 
The main new difficulty is the construction of sequences of functions tending weakly to $0$. 
 Indeed, in this general context, we can neither use Fourier analysis tools like in the proof of Theorem \ref{thm:dirichletseries} nor the existence of inner functions which helped us to construct sequences tending weakly to zero.
 This is this part of the proof which will require that $(\Omega,\mathcal A,\mu)$ is separable.
\begin{proof}[Proof of Theorem \ref{thm:multiplierlebesgue}]
  Let $\Omega=\Omega_d\cup\Omega_a$
 where $\Omega_d\cap\Omega_a=\varnothing,$ $\mu_d=\mu_{|\Omega_d}$ is diffuse and $\mu_a=\mu_{|\Omega_a}$ is purely atomic.
 Let $(A_n)$ be a disjoint sequence of atoms such that $\Omega_a=\bigcup_n A_n$. 
 
 We first show that $\|M_u\|_{e,p\to q}\leq \|u_{|\Omega_d}\|_r.$ For $N\in\mathbb N,$ let us define
 $$u_N=\sum_{n\leq N}a_n\mathbf 1_{A_n}$$
 where $u=a_n$ a.e. on $A_n$ and $K_N=M_{u_N}f.$ Since $f$ is a.s. constant on each $A_n$, $K_N$ is a finite rank operator.
 Hence, it is compact. Now, for any $f\in L^p(\mu),$
 \begin{align*}
  \|M_u f-M_{u_N}f\|_q^q&=\int_{\Omega_d}|uf|^q d\mu+\int_{\bigcup_{n>N}A_n} |uf|^q d\mu\\
  &\leq \|u_{|\Omega_d}\|_r^q\|f\|_p^q+\left(\int_{\bigcup_{n>N}A_n}|u|^r d\mu\right)^q\|f\|_p^q
 \end{align*}
 where we have used Hölder's inequality with $\frac 1p+\frac 1r=\frac 1q.$ Hence, $\liminf_{N}\|M_u-M_{u_N}\|_{p\to q}\leq \|u_{|\Omega_d}\|_r$
 which yields the first inequality.
 
 Conversely, since $(\Omega,\mathcal A,\mu)$ is separable, there exists a sequence $(B_n)$ of subsets of $\Omega_d$ and belonging to $\mathcal A$
 such that, for any $B\in \mathcal A,$ for any $\veps>0,$ one may find $n\geq 1$ such that $\mu(B\Delta B_n)<\veps$. 
 We first construct a sequence $(g_n)$ in $L^p(\mu)$ going weakly to $0.$ Let us fix for a while $n\geq 1$. For $I\subset \{1,\dots,n\}$, $I\neq \varnothing,$
 let us set
 $$C_I=\bigcap_{k\in I}B_k\backslash\left(\bigcup_{k\in I^c}B_k\right).$$
Then the sets $C_I$ are paiwise disjoint. Moreover, for any $k\in\{1,\dots,n\}$, $B_k=\bigcup_{k\in I}C_I.$ If $\int_{C_I}|u|^rd\mu=0$, we set 
$g_n=|u|^r$ on $C_I$. Otherwise, since $|u|^rd\mu_d$ is still a diffuse measure, we may split $C_I$ into a partition $C_I'\cup C_I''$ such that
$$\int_{C_I'}|u|^r d\mu_d=\int_{C_I''}|u|^rd\mu_d=\frac 12\int_{C_I}|u|^r d\mu_d.$$
In that case, we set 
$$g_n=\left\{ 
\begin{array}{ll}
 |u|^{r/p}&\textrm{ on }C_I'\\
 -|u|^{r/p}&\textrm{ on }C_I''
\end{array}\right.
$$
so that $\int_{C_I}g_nd\mu_d=0.$ We finally define $g_n$ on $\Omega_d\backslash \bigcup_{k=1}^n B_k$ by
$$g_n=\left\{ 
\begin{array}{ll}
 |u|^{r/p}&\textrm{ on }\Omega_d\backslash \bigcup_{k=1}^n B_k\\
 0&\textrm{ on }\Omega_a.
\end{array}\right.
$$
We can observe that for any $k\leq n,$ $\int_{B_k}g_nd\mu_d=0$. Hence, for all $k\in\mathbb N,$
$\int_{B_k}g_nd\mu_d$ goes to zero as $n$ tends to $+\infty$. Since $(\mathbf 1_{B_n})_{n\geq 1}$ spans a dense subspace of $L^{p^*}(\mu_d),$
and $g_n=0$ on $\Omega_a$, this ensures that $(g_n)$ goes weakly to $0$ in $L^p(\mu)$.
Hence, 
$$\|M_u\|_{e,p\to q}\geq \limsup_n \frac{\|M_u (g_n)\|_q}{\|g_n\|_p}.$$
Now, $\|g_n\|_p=\|u_{|\Omega_d}\|_r^{r/p}$ and 
$$\|M_u g_n\|_q=\left(\int_{\Omega_d} |u|^{rq} |u|^qd\mu_d\right)^{1/q}=\|u_{|\Omega_d}\|_r^{r/q}$$
so that $\|M_u\|_{e,p\to q}\geq  \|u_{|\Omega_d}\|_r$ as guessed.
\end{proof}

\subsection{The case $p=+\infty$} The proof in this case will share some similarities with that of Proposition \ref{prop:H2}. The key tool will be the use of the conditional expectation. 
The main difference with the previous subsection is that we now work in the target space.
\begin{proof}[Proof of Theorem \ref{thm:multiplierlebesgueinfty}]
The proof of the upper bound is completely similar to that of Theorem \ref{thm:multiplierlebesgue}. Details are left to the reader.
Regarding the lower bound, we may and shall assume that $\Omega=\Omega_d.$ Indeed, if $P$ is the canonical projection $L^q(\Omega,\mu)\to L^q(\Omega_d,\mu_d)$
and $K:L^\infty(\Omega)\to L^q(\Omega)$ is compact,
then $\|M_u-K\|_{\infty\to q}\geq \|M_{u|\Omega_d}-PK\|_{\infty\to q}.$
In the same vein we may and shall assume that $(\Omega,\mathcal A,\mu)$ is a finite measure space. Indeed, writing $\Omega=\bigcup_n\Omega_n$ where $\Omega_n\subset\Omega_{n+1}$
and $\mu(\Omega_n)<+\infty$ for any $n,$ a similar argument shows that
$\|M_u\|_{e,\infty\to q}\geq \|M_{u|\Omega_n}\|_{e,\infty\to q}.$

Let $(B_n)$ be a sequence in $\mathcal A$ such that, for any $B\in\mathcal A,$ for any $\veps>0,$
there exists $n\geq 1$ with $\mu(B\Delta B_n)<\veps.$ Let $\mathcal A_n$ be the $\sigma$-algebra generated by $B_1,\dots,B_n$
and for $f\in L^1(\mu)$, let $\mathcal Q_n(f)=\mathbb E(f|\mathcal A_n)$ be the conditional expectation of $f$ given $\mathcal A_n.$
Each $\mathcal Q_n$ is a contraction of $L^q(\Omega)$ and it is a compact operator. Moreover, for any $f\in L^q(\Omega),$
$\mathcal Q_n(f)$ goes to $f$: this is true if $f$ is a linear combination of step functions and we argue by density of these functions,
using $\|\mathcal Q_n\|\leq 1.$ Let $\mathcal R_N=I-\mathcal Q_n$ which satisfies $\|\mathcal R_n\|\leq 2$ and $(\mathcal R_n)$ converges to $0$ pointwise.
Therefore by Lemma \ref{lem:generalargument}, one obtains
$$\|M_u\|_{e,\infty\to q}\geq \frac 12\limsup_{n\to+\infty} \|\mathcal R_n M_u\|_{\infty\to q}.$$
Now, for $n\geq 1$, $I\subset \{1,\dots,n\}$, $I\neq \varnothing,$
 let us set
 $$C_I=\bigcap_{k\in I}B_k\backslash\left(\bigcup_{k\in I^c}B_k\right).$$
We define a function $g_n$ as follows. 
If $\int_{C_I}|u| d\mu=0$, we set 
$g_n=1$ on $C_I$. Otherwise, since $|u| d\mu$ is still a diffuse measure, we may split $C_I$ into a partition $C_I'\cup C_I''$ such that
$$\int_{C_I'}|u| d\mu=\int_{C_I''}|u| d\mu=\frac 12\int_{C_I}|u| d\mu.$$
In that case, we set 
$$g_n=\left\{ 
\begin{array}{ll}
 1&\textrm{ on }C_I'\\
 -1&\textrm{ on }C_I''.
\end{array}\right. 
$$
We finally define $g_n$ on $\Omega\backslash \bigcup_{k=1}^n B_k$ by $g_n=1$. 
This construction ensures that, for all $A\in\mathcal A_n,$ $\int_A ug_nd\mu=0.$ This yields $\mathcal Q_n M_u g_n=0.$
Now,
\begin{align*}
\|\mathcal R_n M_u\|_{\infty\to q}&\geq \|M_u g_n\|_{\infty\to q}-\|\mathcal Q_n M_u g_n\|_{\infty\to q}\\
&\geq \|M_u g_n\|_{\infty\to q}\\
&\geq \left(\int_{\Omega}|u|^q\right)^{1/q}.
\end{align*}
This finishes the proof of the lower bound $\|M_u\|_{e,\infty\to q}\geq \frac 12 \|u\|_q.$
\end{proof}
When $q=2,$ $\mathcal Q_n$ is an orthogonal projection and $\|\mathcal R_n\|\leq 1$ for all $n\geq 1.$ Therefore we obtain the following corollary:
\begin{corollary}
Let $(\Omega,\mathcal A,\mu)$ be a $\sigma$-finite separable
 measure space and let $u\in L^2(\mu)$. Then $\|M_u\|_{e,\infty\to 2}=\|u_{|\Omega_d}\|_2.$
\end{corollary}

\subsection{The case $1\leq p<q$} Our method also gives the essential norm of $\|M_u\|_{e,p\to q}$ when $1\leq p<q.$ 
The situation here is easier. Indeed, for any $u:\Omega\to\Omega$ measurable, 
$M_u \in \mathcal L(L^p,L^q)$ if and only if $u_{|\Omega_d}=0$ and 
$\sup_n |u(A_n)|/\mu(A_n)^{1/r}<+\infty$ where $r=pq/(p-q)$ and $u$ is a.e. equal to $u(A_n)$ on $A_n$ (see \cite{TaYo99}). 
\begin{proposition}
Let $1\leq p<q$ and set $r=pq/(p-q).$ Let $(\Omega,\mathcal A,\mu)$ be a $\sigma$-finite measure space and let $u:\Omega\to\Omega$ be measurable such that $u_{|\Omega_d}=0$ and $\sup_n |u(A_n)|/\mu(A_n)^{1/r}<+\infty.$ Then 
$$\|M_u\|_{e,p\to q}=\limsup_{n\to+\infty}\frac{|u(A_n)|}{\mu(A_n)^{1/r}}.$$
\end{proposition}
\begin{proof}
Without loss of generality, we can assume that the sequence $(A_n)$ is infinite and $\mu(A_n)\neq 0$ for all $n$ (otherwise, $M_u$ is always compact since it has finite rank). For $s\in\{p,q\},$ denote $\mathcal Q_n^s f=\sum_{k=1}^n \mathbf 1_{A_k}f\in L^s(\mu)$ and $\mathcal R_n^s =\mathrm{Id}_{L^s}-\mathcal Q_n.$ Then $\|\mathcal R_n^s\|=1$ and by Lemma \ref{lem:generalargument},
$$\limsup_n \|\mathcal R_n^q M_u\|\leq \|M_u\|_{e,p\to q}\leq \liminf_n \|M_u \mathcal R_n^p\|.$$
Now,  for any $f\in L^p,$ $\mathcal R_n^q M_u f= M_u \mathcal R_n^p f=\sum_{k=n+1}^{+\infty} u(A_k) \mathbf 1_{A_k}f=: T_n f$. We conclude by \cite[Theorem 1.4]{TaYo99} that 
$$\|T_n\|_{p\to q}=\sup_{k\geq n}\frac{|u(A_n)|}{\mu(A_n)^{1/r}}.$$ 
\end{proof}

\subsection{Weighted composition operators}
In the spirit of \cite{LoLoh20} or of Section \ref{sec:wcoh2} of the present paper, 
our method of proof has applications to weighted composition operators. Let $(\Omega_1,\mathcal A,\mu)$
and $(\Omega_2,\mathcal B,\nu)$ be two $\sigma$-finite measure spaces, let $u:\Omega_2\to \mathbb C$ be measurable
and let $\varphi:\Omega_1\to\Omega_2$ be measurable and nonsingular. The weighted composition operator $uC_\varphi$ is 
defined for $f\in L^p(\mu)$ by
$$uC_\varphi f(x)=u(x)\cdot f\circ\varphi(x),\ x\in\Omega_2.$$
For $q\geq 1,$ the measure $\mu_q$ defined for any $A\in\mathcal A$ by
$$\mu_q(A)=\int_{\varphi^{-1}(A)}|u|^q d\nu$$
is absolutely continuous with respect to $\mu.$ Its Radon-Nikodym derivative will be denoted by $d\mu_q/d\mu.$ It satisfies the important property
$$\|uC_\varphi f\|_{L^q(\nu)}=\|M_{F_{q,u,\varphi}}f\|_{L^q(\mu)}$$
where $F_{q,u,\varphi}=(d\mu_q/d\mu)^{1/q}.$ Then Theorem \ref{thm:multiplierlebesgue} and its proofs yields the following statement.
\begin{theorem}
 Let $(\Omega_1,\mathcal A,\mu)$
and $(\Omega_2,\mathcal B,\nu)$ be two $\sigma$-finite measure spaces with $\Omega_1$ separable, let $u:\Omega_2\to \mathbb C$ be measurable
and let $\varphi:\Omega_1\to\Omega_2$ be measurable and nonsingular. Let finally $p>q\geq 1.$ Then $\|uC_\varphi\|_{e,p\to q}=\|F_{q,u,\varphi|\Omega_{1,d}}\|_r$
where $r=pq/(p-q)$ and $\Omega_{1,d}$ is the diffuse part of $\Omega_1.$
\end{theorem}

\providecommand{\bysame}{\leavevmode\hbox to3em{\hrulefill}\thinspace}
\providecommand{\MR}{\relax\ifhmode\unskip\space\fi MR }
\providecommand{\MRhref}[2]{%
  \href{http://www.ams.org/mathscinet-getitem?mr=#1}{#2}
}
\providecommand{\href}[2]{#2}

\end{document}